\documentclass[12pt]{article}

\RequirePackage{kpfonts}
\RequirePackage[sf,bf,small,raggedright]{titlesec}
\usepackage{dsfont}
\usepackage{comment}
\usepackage{textcomp}
\usepackage{amssymb}
\usepackage{bbm}
\usepackage{fancyhdr}
\usepackage{amsmath}
\usepackage{amsbsy,amsthm}
\usepackage{amscd}
\usepackage{latexsym}
\usepackage{graphicx}   
\usepackage{pdfsync}
\usepackage{blkarray}
\usepackage{multirow}
\usepackage{hyperref}
\usepackage{xcolor}
\usepackage{enumerate}

\usepackage{natbib}
\usepackage{tcolorbox}

\usepackage{tikz}

\usepackage{subcaption}
\usepackage{algorithmic}
\usepackage[justification=centering]{caption}

\usepackage{natbib}
 \def\newblock{\ }%
 \bibpunct[, ]{[}{]}{,}{n}{}{,}%

\usepackage{float}
\usepackage{pgf}
\usetikzlibrary{arrows}

\newcommand{\N}{\mathbb{N}}

\newcommand{\E}{\mathbb{E}}

\newtheorem{lemma}{Lemma}
\newtheorem{theorem}{Theorem}

\newtheorem{proposition}{Proposition}
\newtheorem{definition}{Definition}

\allowdisplaybreaks
\newtheorem{remark}{Remark}

\numberwithin{equation}{section}

\def\IND{\mathbbm{1}}

\newcommand{\EXP}{\mathbb{E}}
\newcommand{\PROB}{\mathbb{P}}

\newcommand{\defeq}{\stackrel{\mathrm{def.}}{=}}

\begin{document}

\title{Archaeology of random recursive dags and 
    Cooper-Frieze random networks
    \thanks{This research was supported by a Huawei Technologies Co., Ltd. grant.
      Simon Briend acknowledges the support of Région Ile de France.
G\'abor Lugosi acknowledges the support of Ayudas Fundación BBVA a
Proyectos de Investigación Científica 2021 and
the Spanish Ministry of Economy and Competitiveness, Grant
PGC2018-101643-B-I00 and FEDER, EU}
\author{
  Simon Briend \\
  Université Paris-Saclay, CNRS, \\
  Laboratoire de Mathématiques d'Orsay, \\
  91405, Orsay, France 
  \and
  Francisco Calvillo  \\
  Department of Mathematics and Applications, \\
  École Normale Supérieure, 75005, Paris, France
\and  
G\'abor Lugosi \\
Department of Economics and Business, \\
Pompeu  Fabra University, Barcelona, Spain \\
ICREA, Pg. Lluís Companys 23, 08010 Barcelona, Spain \\
Barcelona Graduate School of Economics
}
}

\maketitle

\begin{abstract}
We study the problem of finding the root vertex in large growing
networks.
We prove that it is possible to construct confidence sets of size
independent of the number of vertices in the network that contain the
root vertex with high probability in various models of random networks.
The models include uniform random recursive dags and
uniform Cooper-Frieze random graphs. 
\end{abstract}

\section{Introduction}

With the ubiquitous presence of networks in many areas of science and
technology, a  multitude  of new challenges have gained importance in the statistical
analysis of networks. One such area, termed \emph{network
  archaeology} (\citet*{NaKi11})
studies problems about unveiling the past of dynamically growing
networks, based on present-day observations.

In order to develop a sound statistical theory for such problems, one
usually models the growing network by simple stochastic growth dynamics.
Perhaps the most prominent such growth model is the preferential
attachment model, advocated by \citet*{AlBa02}. In these models, vertices of
the network arrive one by one and a new vertex attaches to one or more
existing vertices by an edge according to some simple probabilistic
rule.

Arguably the simplest problem of network archaeology is that of
\emph{root finding}, when one aims at estimating the first vertex of a
random network, based on observing the (unlabeled) network at a much later point
of time. 

The existing literature on the theory of network archaeology
mostly focuses on the simplest possible kind of networks, namely
trees, 
see
\citet*{Hai70},
\citet*{ShZa11,ShZa16},
\citet*{BuMoRa15},
\citet*{CuDuKoMa15},
\citet*{KhLo16},
\citet*{JoLo16,JoLo17a},
\citet*{BuElMoRa17},
\citet*{BuDeLu17},
\citet*{LuPe19},
\citet*{ReDe19},
\citet*{BaBh20},
\citet*{CrXu21},
\citet*{AdDeLuVe21},
\citet*{BrDeGo21}.

In various models of growing random trees, it is
quite well understood up to what extent one may identify the origin of
the tree (i.e., the root) by observing a large unlabeled tree. These
models include uniform and linear preferential attachment trees and
diffusion over regular trees. Remarkably, in all these models,
the size of the tree does not play a role. In other words, there exist
root-finding algorithms that are able to select a small number of
nodes -- independently of the size of the tree -- such that the root
vertex is among them with high probability.

Here we address the more difficult -- and more
realistic -- problem of finding the origin of growing networks when
the network is not necessarily a tree. The added difficulty stems from the fact
that the centrality measures that proved to be successful in root
estimation in trees crucially rely on properties of trees.

A notable exception in the literature is the recent paper of \citet*{CrXu21a} in which
the authors allow for a ``noisy'' observation of the tree. In their
model, the union of the tree of interest and an (homogeneous)  Erd\H{o}s-R\'enyi
random graph is observed, and the goal is to estimate the root of
the tree.

In this paper we study root estimation in two more complex network models.
Both of these models may be viewed as natural extensions of the
random recursive trees that were in the focus  of most of the previous study of network
archaeology. 
Recall that a uniform random recursive tree on the vertex
  set $[n]$ is defined recursively, such that each vertex $i\in \{2,3,\ldots,n\}$
  is attached by an edge to a vertex chosen uniformly at random among
  the vertices $\{1,\ldots,i-1\}$, see, e.g., \citet*{Drm09}.

In particular, we study the problem of root finding in
(1) uniform random recursive dags; and
(2) uniform Cooper-Frieze random graphs.

\subsubsection*{Uniform random recursive dags}

For a positive integer $\ell$, a uniform random $\ell$-dag is simply
the union of $\ell$ independent uniform  random recursive trees on
the same vertex set $[n]$. Equivalently, a uniform random $\ell$-dag
may be generated recursively; each vertex $i\in \{2,3,\ldots,n\}$
  is attached by an edge to $\ell$ vertices chosen uniformly at
  random  (with replacement) among  the vertices $\{1,\ldots,i-1\}$.
Multiple edges are collapsed so that the resulting graph is simple. 
Random recursive dags have been studied by
\citet*{DiSeSpToTs94,TsXh96,TsMa01,DeJa11,BrFa12,Mah14}, among others.

\begin{definition}
Let $n,\ell\in \N$.
For $i=[\ell]$, let $G_i=(V,E_i)$ be independent uniform random recursive trees on the vertex set
$V=[n]$.
A \emph{uniform random recursive $\ell$-dag} on $n$ vertices
is $G=(V,E_1\cup \cdots \cup E_\ell)$.
\end{definition}

\subsubsection*{Uniform Cooper-Frieze random graphs}

The other network model studied here was introduced by 
\citet*{CoFr03} in an attempt to mathematically describe large web graphs, see also \citet*{KaFr16}. 
In the Cooper-Frieze network model both vertices and edges are added
sequentially to the network based on uniform or preferential
attachment mechanisms. 
The model is quite general but here we focus on the simplest version
when both vertices and edges are added by \emph{uniform} attachment.

More precisely, the uniform Cooper-Frieze growth model is defined
as follows. The procedure has a parameter $\alpha \in (0,1)$.
The process is initialized by a graph containing a single vertex and
no edges. At each time instance $t=1,2,\ldots$, an independent
Bernoulli$(\alpha)$ random variable $Z_t$ is drawn. If $Z_t=0$,
a new vertex is added to the vertex set along with an edge that
connects this vertex to one of the existing vertices, chosen uniformly
at random. If $Z_t=1$, then
a new edge is added by choosing two existing vertices uniformly at
random and connecting them. Note that the resulting graph may have
multiple edges. In such cases, we may convert the graph into a simple
graph by keeping only one of each multiplied edge.

If one runs the process for $T$ steps for a large
value of $T$, the graph has
$n \sim \text{Binomial}(T-1,1-\alpha) \approx (1-\alpha)T$ vertices and $T-1 \approx n/(1-\alpha)$
edges. If one removes the edges added at the times when $Z_t=1$,
the remaining graph is a tree, distributed as a uniform random recursive
tree on $n$ vertices.
The remaining $T-n-1$ edges are present
approximately independently of each other and there is an edge between
vertices $i$ and $j$ (where $1\le i< j \le n$) if
\[
    \sum_{t =1}^T   \sum_{\ell=j}^n \IND_{t\in \{t_{\ell}+1,t_{\ell+1}-1\}}
    \IND_{\text{the pair $(i,j)$ is selected at time $t$}} \ge 1~,
\]
where $1=t_1< t_2 < \cdots < t_n\le T$ are the times when new vertices
are added, that is, when $Z_t=1$. Since the probability that edge 
$(i,j)$ is selected at time $t\in \{t_{\ell}+1,t_{\ell+1}-1\}$ is
$1/\binom{\ell}{2}$, for large values of $T$, the probability that
edge $(i,j)$ is present in the graph after $T$ steps is concentrated
around
\[
  \frac{c_{\alpha}}{\max(i,j)-1} \quad \text{where} \quad  c_\alpha \defeq \frac{2}{1-\alpha}~,
\]
whenever $\max(i,j)-1\ge c_{\alpha}$.
Hence, the uniform Cooper-Frieze model is essentially equivalent to
the following random graph model. In order to avoid some tedious and uninteresting 
technicalities, we work with this modified model instead of the
original recursive definition.

\begin{definition}
Let $n\in \N$ and let $c$ be a positive constant.
Let $G_1=(V,E_1)$ be a uniform random recursive tree on the vertex set
$V=[n]$.
Let $G_2=(V,E_2)$ be a random graph on the same vertex set,
independent of $G_1$, such that edges of $G_2$ are present
independently of each other, such that for all $i\neq j$,
\[
  \PROB\{(i,j) \in E_2\} = \min\left(\frac{c}{\max(i,j)-1},1 \right)~.
\]
Finally, the \emph{uniform Cooper-Frieze random graph} with parameters
$c$ and $n$ is $G=(V,E_1\cup E_2)$.
\end{definition}

\subsubsection*{Root estimation}

The main result of this paper is that finding the root is possible
both in uniform random recursive dags and
in
uniform Cooper-Frieze random graphs. More precisely, one may find
\emph{confidence sets} for the root vertex whose size does not depend
on the number of vertices in the graph.
To make such statements rigorous, consider the following definition.

\begin{definition}
  Let $\{G^{(n)}\}$ be a sequence of random graphs such that
  $G^{(n)}$ has vertex set $[n]$. We say that
  \emph{root estimation} is possible if the following holds.
  For every $\epsilon >0$, there exists a positive integer
  $K(\epsilon)$
  such that, for every $n\in \N$, upon observing the graph $G^{(n)}$
  without the vertex labels, one may find a set $S \subset [n]$ of vertices
  of size $|S|=K(\epsilon)$ such that
  \[
    \PROB\{1\in S\} \ge 1- \epsilon~.
  \]
\end{definition}

The set $S$ in the above definiton is often called a confidence set
for the root vertex. 

As mentioned above, root estimation has mostly been studied for
random recursive trees. \citet*{BuDeLu17} show that root estimation
is possible in the uniform random recursive tree and linear
preferential attachment trees. They show that in the case of the
uniform random recursive tree, one may take
$K(\epsilon) \le
\exp\left(c\log(1/\epsilon)/\log\log(1/\epsilon)\right)$
for some constant $c$.
For linear preferential attachment trees one may take
$K(\epsilon)=c\epsilon^{-2-o(1)}$, as shown by \citet*{BaBh20}
who also show that root estimation is possible for a wide class of
preferential attachment trees. Building on the papers of
\citet*{ShZa11,ShZa16}, \citet*{KhLo16} show that root estimation
is possible in random trees obtained by diffusion on an infinite regular 
tree, and that one my take
$K(\epsilon)  =\exp\left(O\left(\log(1/\epsilon)/\log\log(1/\epsilon)\right)\right)$.
\citet*{BrDeGo21} study root estimation in size-conditioned Galton–Watson
trees.

The sets $S$ of constant size that establish the possibility of root
estimation for various trees usually contain the set of most
``central'' vertices according to some notion of centrality such as
\emph{Jordan centrality} (as in \cite{BuDeLu17}, \cite{BaBh20}) or
\emph{rumor centrality} introduced in \cite{ShZa11,ShZa16}, see also
\cite{BuDeLu17}, \cite{KhLo16}. However, these notions are suited for
trees only and when the observed network is more complex, new ideas
need to be introduced. \citet*{CrXu21a} study a model in which the
observed network consists of either a uniform attachment tree (i.e.,
uniform random recursive tree) or a preferential attachment tree, with
random edges added (independently over all possible vertex pairs, with
the same probability). They introduce a Bayesian method and prove that
it is able to estimate the root as long as there are not too many
edges, where the threshold value depends on the particular model.
It is unclear if the method of \cite{CrXu21a} may be generalized to
the random graph models studied here. Instead, we
introduce an alternative root estimation method that is based
on the appearance of certain subgraphs.

The main results of this paper are summarized in the following two theorems.

\begin{tcolorbox}
\begin{theorem}
\label{thm:main1}
Let $G=G^{(n)}$ be a uniform random $\ell$-dag on $n$
vertices. Root estimation is possible in $G$. In
particular, there exist numerical constants $c_0,c_1,c_2>0$
such that, whenever $\epsilon \le e^{-c_2\ell}$,
one may take
\[
  K(\epsilon) \le \frac{c_0}{\epsilon} \log(1/\epsilon)^{\frac{c_1}{\ell}\log(1/\epsilon)}~.
\]
\end{theorem}
\end{tcolorbox}

Explicit values of the constants $c_0,c_1,c_2$ are  given in the proof below.
In the uniform Cooper-Frieze model we have a similar bound:

\begin{tcolorbox}
\begin{theorem}
\label{thm:main2}
Let $G=G^{(n)}$ be a uniform Cooper-Frieze random graph on $n$
vertices, with parameter $c$. Root estimation is possible in $G$. In
particular, one may take
\[
  K(\epsilon) \le c_0\log(1/\epsilon)^{c_1\log(1/\epsilon)}
\]
for some constants $c_0,c_1>0$ depending only on $c$.
\end{theorem}
\end{tcolorbox}

The main results establish that, upon observing the graph after removing
its vertex labels, one may find a set $S$ of vertices of size independent
of $n$ such that $S$ contains the root vertex (i.e., vertex $1$) with
probability at least $1-\epsilon$. The size of the set is bounded by
a function of $\epsilon$ only.

Observe that if $\ell$ is of the order of
$\log(1/\epsilon)$, then the bound for $K(\epsilon)$ is $1/\epsilon$ times a
poly-logarithmic term in $1/\epsilon$. 
On the other hand, when $\ell$ is a fixed constant, as $\epsilon \to 0$,
the obtained bounds are super-polynomial
in $1/\epsilon$, significantly larger than the analogous bounds
obtained for uniform and preferential attachment trees.
In all ranges of $\ell$, these bounds are inferior to the best upper bounds 
available for the case $\ell=1$ (i.e., uniform random recursive trees).
We do not claim optimality of this bound. It is an interesting open
question whether much smaller vertex sets may be found with the
required guarantees.
We conjecture that for any $\ell >1$, root finding is easier in a uniform random $\ell$-dag than in
a uniform random recursive tree. If that is the case, one should be able to take $K(\epsilon)$
as $\exp\left(O\left(\log(1/\epsilon)/\log\log(1/\epsilon)\right)\right)$.
Similar remarks hold for the bound of Theorem \ref{thm:main2}.

In order to prove Theorems \ref{thm:main1} and \ref{thm:main2},
we propose a root estimation procedure and prove that the same
procedure works in both models. The procedure looks for certain
carefully selected subgraphs that we call \emph{double cycles}.
The set $S$ of candidate
vertices are certain special vertices of such double cycles.

The rest of the paper is organized as follows. In Section \ref{sec:dc}
we introduce the proposed root estimation procedure. The proof of
Theorem \ref{thm:main1} is given in Section \ref{sec:proofofthm1}
while Theorem \ref{thm:main2} is proved in Section
\ref{sec:proofofthm2}.

\section{Double cycles}
\label{sec:dc}

In this section we define the root estimation method that we use to
prove the main results. In order to determine the set $S$ of vertices
that are candidates for being the root vertex, we define ``double cycles''.

Let $s,t$ be positive integers.
We say that a vertex $v\in [n]$ is an \emph{anchor of a double cycle}
of size $(s,t)$ if there exists an integer $0<p\le  \min(s,t)/2$ and
$s+t-1-p$ different vertexes $i_1,i_2,\dots i_{s+t-2-p} \in [n]$, such that

\noindent
$\bullet$
vertices $v,i_1,\ldots,i_{s-1}$ form a cycle of length $s$ in $G$ (in
this order);

\noindent
$\bullet$
vertices $v,i_{s+1-p},\ldots,i_{s+t-1-p}$ form a cycle of length $t$
in $G$ (in this order).

Note that the two cycles are disjoint, except for the common path $v\sim \cdots \sim i_{p-1}$ (so $p$ is the number of common vertices in both cycles).
Also note that $i_{p-1}$ is another anchor of the same double cycle.
If $p=1$, we declare $i_0=v$. In that case the two cycles intersect in
the single vertex $v$ and the double cycle has a unique anchor $v$,
see Figure \ref{fig:doublecycle1}.

In other words, if two vertices $v,u\in [n]$ are connected by three
disjoint paths such that the sum of the lengths of the first and
second paths is $s$ and the sum of the lengths of the second and
third paths is $t$, then $v$ and $u$ are anchors of a double cycle of
size $s$ and $t$. Also, $v$ is the anchor of a double cycle
of size $(s,t)$ if vertex $v$ is the unique common vertex of two cycles of lengths $s$ and $t$.

\begin{figure}[H]
\includegraphics[scale=1.0]{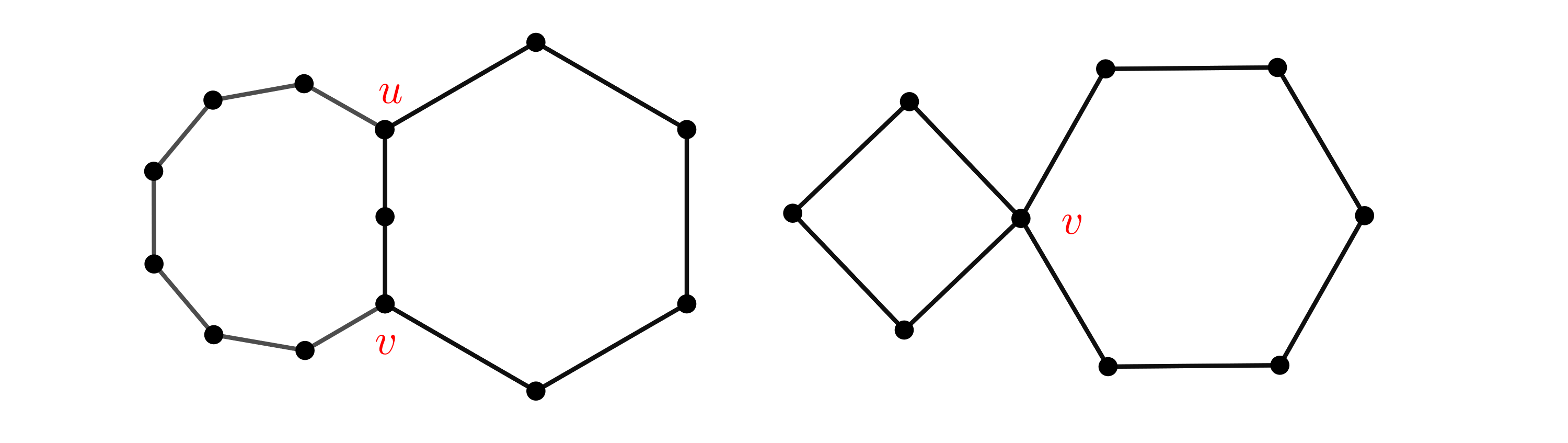}
\caption{Examples of double cycles}
\label{fig:doublecycle1}
\end{figure}

For a positive integer $m$, let $S_m \subset [n]$ be the set of
vertices $i$ such that $i$ is an anchor of a double cycle of size
$(s,t)$ for some $s\le m$ and $t \le m$.

In order to prove Theorem \ref{thm:main1}, it suffices to show that for any given
$\epsilon \in (0,1/100)$, one may take
$m=m_\epsilon =\left\lceil\frac{30}{\ell}\log(1/\epsilon)\right\rceil$ such that
\[
  \PROB\left\{ 1 \in S_m \ \text{and} \ |S_m| \le \frac{4}{\epsilon} \ell^{2m}(2m)! \right\} \ge 1-\epsilon~.
\]
This follows if we prove that we have both
\begin{equation}
\label{eq:rootinset}
 \PROB\left\{ 1 \in S_m \right\} \ge 1-\frac{\epsilon}{2}
\end{equation}
and
\begin{equation}
\label{eq:setissmall}
 \PROB\left\{ |S_m| \le \frac{4}{\epsilon} \ell^{2m}(2m)! \right\} \ge 1-\frac{\epsilon}{2}~.
\end{equation}
We prove \eqref{eq:rootinset} in Section \ref{sec:root} and \eqref{eq:setissmall} in Section \ref{sec:largeindices}. 

\begin{remark}
The reader may wonder why the proposed method looks for double cycles
as opposed to simpler small subgraphs such as triangles or a clique of
size $4$ with an edge removed, etc. The reason is that such simpler
subgraphs are either too abundant in the sense that vertices with high
index may be contained in (many of) them or the root vertex is not
contained in any of them with some probability that is bounded away
from zero. This may happen in
spite of the fact that the expected number of such small subgraphs
containing the root vertex goes to infinity as $n\to \infty$. Double cycles guarantee
the appropriate concentration expressed in \eqref{eq:rootinset}.
\end{remark}

\section{Proof of Theorem \ref{thm:main1}}
\label{sec:proofofthm1}

As it is explained in the previous section, in order to prove Theorem \ref{thm:main1}, it is enough to
prove the inequalities \eqref{eq:rootinset} and \eqref{eq:setissmall}, where $S_m$ is the set 
of those vertices that are anchors of a double cycle of size $(s,t)$ for some $s,t\le m$.

\subsection{The root vertex is the anchor of a small double cycle}
\label{sec:root}

First we consider the case when $\ell=2$. Then the observed graph $G$ is
the union of two independent random recursive trees $T_1$ and
$T_2$. To prove \eqref{eq:rootinset} we need to ensure that vertex $1$
is the anchor of a double cycle of small size, with probability at
least $1-\epsilon/2$. To do so, it suffices to show that
there are two edges $(1,i)$ and $(1,j)$ that are present in $T_2$ but not
in $T_1$ where $i$ and $j$ are ``small''-- whose meaning is specified below. 
Indeed, in this case there are two cycles
containing vertex $1$ formed as follows:
\begin{itemize}
\item the unique path from vertex $1$ to $i$ in $T_1$ loops back to
  $1$ thanks to edge $(1,i)$, present in $T_2$;
\item the unique path from vertex $1$ to $j$ in $T_1$ loops back to
  $1$ thanks to edge $(1,j)$,  present in $T_2$.
\end{itemize}
The only intersection of those two cycles is the intersection of the
paths in $T_1$ from vertex $1$ to $i$ and from vertex $1$ to $j$. In a
tree, the intersection of two paths is either empty or a path
itself. Here the intersection is not empty since both paths contain vertex $1$. Thus,
vertex $1$ is in two cycles which only intersect in a path having
vertex $1$ as an extremity, meaning that vertex $1$ is the anchor of a
double cycle. Next we show that two such edges indeed exist, with high probability.

For a vertex $i\in [n]$, the probability that the edge $(1,i)$ is present in
$T_2$ is $1/(i-1)$. The probability that it is absent in $T_1$ is
$1-1/(i-1)$. By independence of $T_1$ and $T_2$, the probability
that the edge $(1,i)$ is present in $T_2$ and absent in $T_1$ is
$\left(1-1/(i-1)\right)/(i-1)$. Let $X_k$ denote the
 number of edges of form $(1,i)$ for some $i\in [k]$ that
are not edges in $T_1$. Then $X_k$ may be written as a sum 
of independent random variables,
\[
  X_k \ = \ \sum_{i=2}^k B_i
\]
where $B_i$ is a Bernoulli random variable with parameter
$\frac{1}{i-1}\left(1-\frac{1}{i-1}\right)$.

If $X_k\ge 2$, there exist two edges of form $(1,i)$ with $i\le k$ that are not present in $T_1$.
By a standard bound for the lower tail for
for sums of nonnegative independent random
variables, see \cite[Exercise 2.9]{BoLuMa13}, we have
\[
  \PROB\left\{ X_k \geq 2 \right\} \ \geq \ 1 - \exp\left(
    -\frac{\left(\E[X_k]-1 \right)^2}{2\E[X_k]} \right)~.
\]  
Since $\E[X_k]$ is easily seen to fall between $\log(k)-2$ and
$\log(k)-1$, we have
\[
  \PROB\left\{ X_k \geq 2 \right\} \ \geq \ 1 - \exp\left( -\frac{1}{2}\log(k)+\frac{5}{2} -\frac{2}{\log(k)-1} \right)~. 
\]
Hence, for $k_{\epsilon}=\left\lceil 16e^5/\epsilon^2\right\rceil$, we have
$\PROB\{X_{k_{\epsilon}}\geq 2\}\geq 1-\epsilon/4$. This implies
that, with probability at least $1-\epsilon/4$, vertex $1$ is the anchor of a double cycle such that all vertices in the double cycle are in $[k_{\epsilon}]$. To conclude the proof of
\eqref{eq:rootinset} we need to check that indeed the size of the
double cycle containing vertex $1$ is at most $m$. Such double cycles are formed by
a path in $T_1$, closed by an additional edge coming from
$T_2$. Therefore, both cycles contained in the double
cycle of interest have a size bounded by the height of the subtree of
$T_1$ induced by the vertex set $[k_{\epsilon}]$, plus
$1$. By well-known bounds for the height of a uniform random recursive
tree (see, e.g.,  Drmota \cite{Drm09}, Devroye \cite{treedevroye},
Pittel \cite{pittel}) we
have that the
depth of a uniform random recursive tree on $k$ vertices is bounded
by $e\log(k)+e\log(4e/\epsilon)$ with probability at least
$1-\epsilon/4$, see  Drmota \cite[p. 284]{Drm09}.

Plugging in the value of $k_\epsilon$, we get that for any
$\epsilon\leq 10^{-2}$, the diameter of a uniform recursive random
tree of size $k_{\epsilon}$ is at most $15\log(1/\epsilon)$, with
probability at least $1-\epsilon/4$.

Putting these bounds together, we have that, in the
case $\ell=2$, with probability at least $1-\epsilon/2$,
vertex $1$ is an anchor of a double cycle of size $(s,t)$ with
$s,t \le 15\log(1/\epsilon)$,  implying \eqref{eq:rootinset} for $\ell=2$.

It remains to extend the above to the general case of $\ell\geq
2$.
Since $G$ is the union of $\ell$ independent uniform random recursive trees, it contains
the union of $\lfloor \ell/2\rfloor$ independent, identically distributed random uniform $2$-dags.
Using the result proved for random uniform
$2$-dags above, 
the probability than in
$G$, vertex $1$ is not the anchor of a double cycle of size at most
$15\log(\epsilon^{2/(\ell-1)})$ is at most $\epsilon$. This
concludes the proof of \eqref{eq:rootinset} in the general case.

\subsection{High-index vertices are not anchors of double cycles}
\label{sec:largeindices}

In order to prove \eqref{eq:setissmall} we need to show that no vertex
with high index is the anchor of a double cycle of size smaller than
$m$. We bound the probability that there exists $v
>K$ such that  $v \in S_m$, where recall that $K=K(\epsilon)$.
To this end, we count $C_{s,t}(v)$, the
number of double cycles of size $(s,t)$ having vertex $v$ as an
anchor. Then, by the union bound,
\begin{equation}
  \label{unionboundSM}
  \PROB \left\{ \exists v>K: v \in S_m\right\}
  \leq  \sum_{v\geq K} \sum_{s,t \leq m_{\epsilon}}
  \PROB\{C_{s,t}(v)\geq 1\}
   \leq  \sum_{v\geq K} \sum_{s,t \leq m_{\epsilon}} \E C_{s,t}(v)~.
\end{equation}
In order to bound $\E C_{s,t}(v)$,
we may assume, without loss of generality, that $s \leq t$.


For a permutation $\sigma\in \Pi_{s+t-2-p}$,
we denote by $C(s,t,p,v,\sigma, i_1,\ldots, i_{s+t-p-2})$
the following event:

\begin{itemize}
\item if $p=1$,
\begin{eqnarray*}
 \lefteqn{   
    C(s,t,1,v,\sigma, i_1,\ldots, i_{s+t-2})     }   \\
  &  = & \left\{v \sim i_{\sigma(1)} \sim \cdots \sim
  i_{\sigma(s-1)}\sim v \sim i_{\sigma(s)} \sim \cdots \sim i_{\sigma(s+t-2)} \sim v\right\}
  ~,
\end{eqnarray*}  
\item and if $p>1$
\begin{eqnarray*}
 \lefteqn{   
  C(s,t,p,v,\sigma, i_1,\ldots, i_{s+t-p-2})     }   \\
  & = & \left\{v \sim i_{\sigma(1)} \sim \cdots \sim
  i_{\sigma(s-1)}\sim v \sim i_{\sigma(s)} \sim \cdots \sim
  i_{\sigma(s+t-2-p)} \sim i_{\sigma(s-p)}\right\}~.
\end{eqnarray*}  
\end{itemize}
where $i\sim j$ denotes that vertices $i$ and $j$ are joined by an
edge.
Thus, $C(s,t,p,v,\sigma, i_1,\ldots, i_{s+t-p-2})$ is
the event that the double cycle of size $s,t$ ($s\leq t$) having
$p$ vertices in the intersection, with $v$ as an anchor and on the set
of vertices $\{i_1,\ldots,i_{s+t-p-2}\}$ ordered by $\sigma$ as
illustrated in Figure \ref{fig:doublecycle2} is present.

\begin{figure}[H]
\begin{center}
\includegraphics[scale=1]{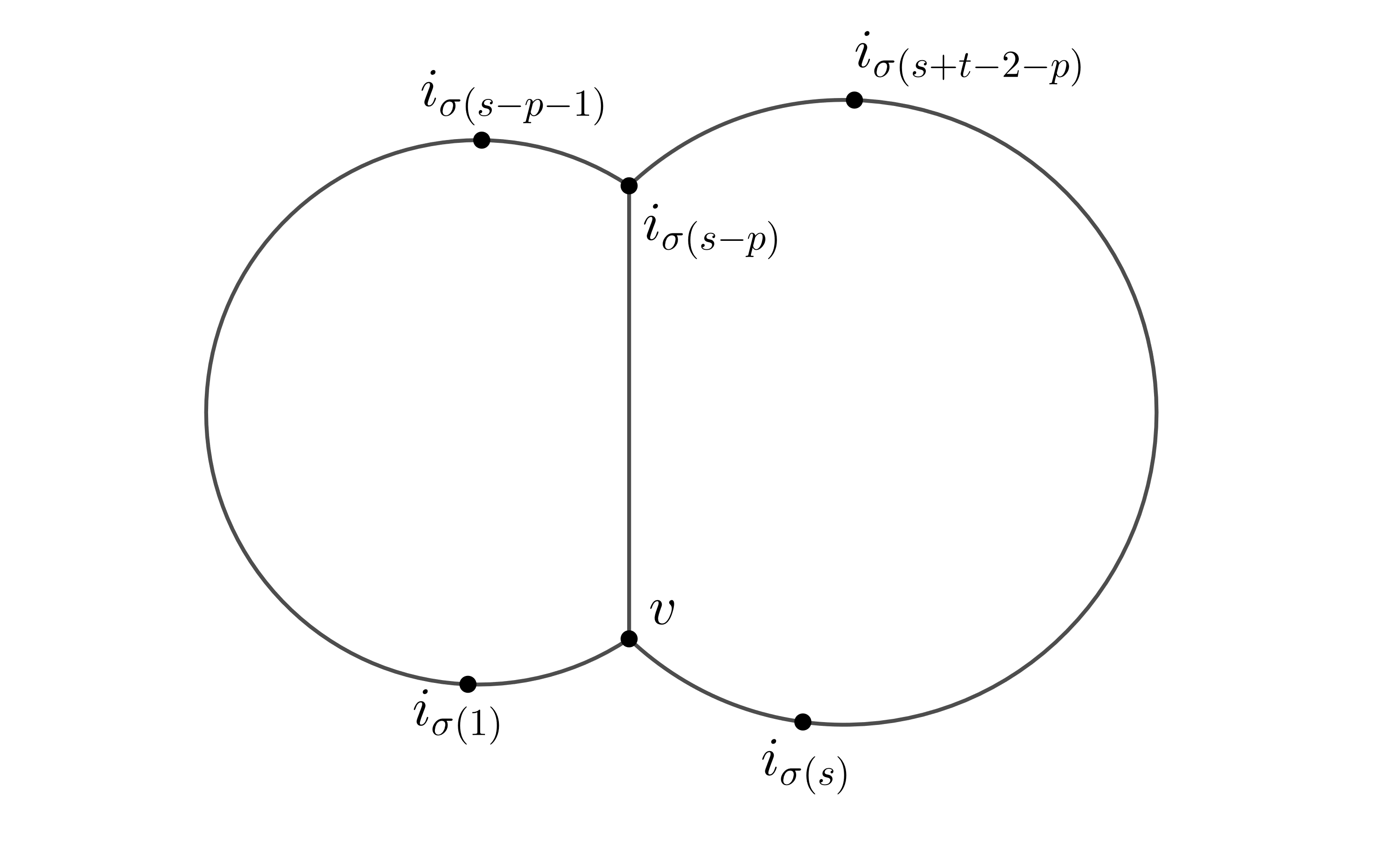}
\end{center}
\caption{Index ordering in a double cycle}
\label{fig:doublecycle2}
\end{figure}

With this notation, we may write $C_{s,t}(v)$ as follows:
\begin{equation}
\begin{aligned}
\label{eq:doublecycles}  
C_{s,t}(v)   
 = &  \sum_{p=1}^{\lfloor s/2\rfloor} \ \ \  \sum_{i_1<\dots< i_{s+t-2-p}} \ \ \   \sum_{\sigma\in
  \Pi_{s+t-2-p}}  \IND_{C(s,t,p,v,\sigma, i_1,\ldots, i_{s+t-p-2})}~,
\end{aligned}
\end{equation}
in order to bound the expected number $\EXP C_{s,t}(v)$ of
double cycles of size $(s,t)$ anchored at $v$, we need to estimate $\PROB\left\{C(s,t,p,v,\sigma, i_1,\ldots, i_{s+t-p-2})\right\}$.

This exact probability is difficult to compute.
Instead, we make use of the following proposition that establishes that a
uniform random $\ell$-dag is dominated by an appropriately defined
inhomogeneous Erd\H{os}-R\'enyi random graph. This random graph is
defined as a graph on the vertex set $[n]$ such that each edge is
present independently of the others and the probability that vertex
$i$ and vertex $j$ are
connected by an edge equals
\[
    \pi(i,j) \defeq \min\left(1, \frac{\ell}{\max(i,j)-1} \right)~.
\]
The next proposition shows that every fixed subgraph is at most as likely to
appear in a uniform random $\ell$-dag as in the inhomogeneous 
Erd\H{os}-R\'enyi random graph.

\begin{proposition}
\label{lem:negativecorrelation}
Let $G=(V,E)$ be a uniform random $\ell$-dag on the vertex set
$V=[n]$. For some $k\le \binom{n}{2}$, let
$(a_1,b_1),\ldots,(a_k,b_k)$
be distinct pairs of vertices such that $a_i\neq b_i$ for all $i\le k$.
Then
\[
  \PROB\left\{ (a_1,b_1),\ldots,(a_k,b_k)\in E \right\}\leq \prod_{i=1}^k \pi(a_i,b_i)~. 
\]
\end{proposition}

\begin{proof}
  Recall that the edge set of $G$ may be written as
  $E= \cup_{j=1}^{\ell} E_j$, where
$(V,E_1),\ldots, (V,E_{\ell})$ are independent uniform random recursive
trees.
We may assume, without loss of generality, that $b_i > a_i$ for all $i\in [k]$.

We prove the proposition by induction on $k$. For $k=1$, the inequality
follows from the union bound: 
\begin{equation}
\label{eq:induction1}
\PROB\left\{ (a_1,b_1)\in E \right\}\leq \sum_{j=1}^{\ell} \PROB\left\{(a_1,b_1)\in E_j\right\} = \frac{\ell}{\max(a_1,b_1)-1} ~.
\end{equation} 
For the induction step, suppose the claim of the proposition holds for
up to
$k$ edges and consider $k+1$ distinct pairs
$(a_1,b_1),\ldots,(a_{k+1},b_{k+1})$.
Then, by the induction hypothesis, 
\begin{eqnarray*}
  \lefteqn{
  \PROB\left\{ (a_1,b_1),\ldots,(a_{k+1},b_{k+1})\in E \right\}    } \\
& = & \PROB\left\{ (a_1,b_1),\ldots,(a_k,b_k)\in
      E \right\}  \ \PROB\left\{  (a_{k+1},b_{k+1})\in E \ | \
      (a_1,b_1),\ldots,(a_k,b_k)\in E \right\}\\
  & \leq & \PROB\left\{  (a_{k+1},b_{k+1})\in E \ | \
      (a_1,b_1),\ldots,(a_k,b_k)\in E \right\} \prod_{i=1}^k \pi(a_i,b_i)~.
\end{eqnarray*}
Thus, it suffices to show that for all pairs
$(a_1,b_1),\ldots,(a_{k+1},b_{k+1})$,
\[
  \PROB\left\{  (a_{k+1},b_{k+1})\in E \ | \
    (a_1,b_1),\ldots,(a_k,b_k)\in E \right\} \le \pi(a_{k+1},b_{k+1})~.
\]
First, consider the simpler case when for all $i \in [k]$,
$b_i\neq b_{k+1}$. Then, for every fixed $j\in [\ell]$, 
the events $\{(a_1,b_1)\in E_j,\ldots,(a_k,b_k)\in E_j\}$ and
$\{(a_{k+1},b_{k+1})\in E_j\}$ are
independent.
Moreover since the $\ell$ uniform random recursive trees
are independent, the events $\{(a_1,b_1)\in E,\ldots,(a_k,b_k)\in E\}$ and
$\{(a_{k+1},b_{k+1})\}\in E$ are also independent, and therefore
\[
  \PROB\left\{ (a_{k+1},b_{k+1})\in E
  ~|~ (a_1,b_1),\ldots,(a_k,b_k)\in E
\right\}=\PROB\left\{ (a_{k+1},b_{k+1})\in E\right\} \le \pi(a_{k+1},b_{k+1})~,
\]
by \eqref{eq:induction1}.

Now, suppose that there exist some
$i\in [k]$ such that $b_i=b_{k+1}$. We may assume that there exists
a $w\in [k]$ such that
$b_1,\ldots, b_w=b_{k+1}$ and for all $i\in [w+1,k]$,
$b_i\neq b_{k+1}$.
Since each $(V,E_j)$ is a recursive tree, 
$(a_i,b_{k+1}) \in E_j$ and $(a_{k+1},b_{k+1}) \in E_j$ cannot happen at the
same time.
Thus, edge $(a_{k+1},b_{k+1})$ can only be present in the sets $E_j$
that do not contain any of the edges $(a_i,b_{k+1})$. Hence, introducing
$A=\#\left\{ j\in[\ell]: \ E_j\cap \{
  (a_1,b_{k+1}),\ldots,(a_w,b_{k+1}) \} \neq \emptyset \right\}$, we
have, for all $a\in [\ell]$,
\[
  \PROB\left\{  (a_{k+1},b_{k+1})\in E ~|~ (a_1,b_1),\ldots,(a_k,b_k)\in
    E \ \text{and} \ A=a \right\} \ = \
  \PROB\left\{ (a_{k+1},b_{k+1})\in \cup_{j=1}^{\ell-a}E_j
    \right\}~.
\]  
Using the union bound again,
\[
  \PROB\left\{ (a_{k+1},b_{k+1})\in \cup_{j=1}^{\ell-a}E_j
     \right\} \leq \frac{\ell-a}{b_{k+1}-1}\leq
  \frac{\ell}{b_{k+1}-1}  ~.
\]  
Since this holds for all $a$, we have
\[
  \PROB\left\{  (a_{k+1},b_{k+1})\in E ~|~ (a_1,b_1),\ldots,(a_k,b_k)\in
    E \right\} \leq
  \frac{\ell}{b_{k+1}-1}~,
\]  
as desired.
\end{proof}

To count $C_{s,t}(v)$ we split the sum in \eqref{eq:doublecycles} by adding a parameter $r$ 
in order to
 separate the vertices $i_1,\ldots, i_{s+t-2-p}$  according to whether they are smaller
or larger than $v$, obtaining
\[
  C_{s,t}(v)   =  \sum_{p=1}^{\lfloor s/2\rfloor}   \ \ \ \ \sum_{r=0}^{s+t-p-2}  \ \ \ \ \sum_{\sigma\in \Pi_{s+t-2-p}} 
\sum_{i_1<\cdots< i_r<v} \ \ \ \ 
 \sum_{v<i_{r+1}<\cdots< i_{s+t-2-p}} \IND_{C(s,t,p,v,\sigma, i_1,\ldots, i_{s+t-p-2})}~.
\]
From Proposition \ref{lem:negativecorrelation} we know that the
probability of each given double cycle is upper bounded by the product
of $\pi(i,j)=\ell/(max(i,j)-1)$. 
Thus we introduce
$E_{\sigma}(j)\in\{0,1,2,3,4\}$ counting the number of vertices
neighboring vertex $i_j$ in the double cycle, that have indices
smaller than $i_j$. By convention we write $E_{\sigma}(0)$ for the
analogous quantity for vertex $v$. Doing so, we may write
\begin{equation}\label{expectationcalculus}
\begin{aligned}
\EXP C_{s,t}(v)  & \leq  \sum_{p=1}^{\lfloor s/2\rfloor} \ \  \ell^{s+t-p} \ \  \ \sum_{r=0}^{s+t-p-2}   \ \ \ \ \sum_{\sigma\in \Pi_{s+t-2-p}} \ \ \ (v-1)^{-E_{\sigma}(0)}  \\
& \ \ \ \times \left( \sum_{i_1<\cdots < i_r<v} \ \ \prod_{j=1}^r
  (i_j-1)^{-E_{\sigma}(j)} \right)\times \left( \sum_{v<i_{r+1}<\cdots
    < i_{s+t-2-p}} \ \  \prod_{j=r+1}^{s+t-2-p} (i_j-1)^{-E_{\sigma}(j)} \right) ~.
\end{aligned}
\end{equation}
This allows us to decompose the sum in two parts; the sum involving
the $r$ vertices with index
smaller than $v$ and the $s+t-2-p-r$ vertices with index larger than $v$. If we fix
$p$, $m$ and $\sigma$, we need to upper bound both
\[
  A(\sigma,p,r):=A= \sum_{i_1<\cdots i_r<v} \ \ \prod_{j=1}^r
  (i_j-1)^{-E_{\sigma}(j)}
\]
and
\[
  B(\sigma,p,r):=B=\sum_{v<i_{r+1}<\cdots< i_{s+t-2-p}} \ \
  \prod_{j=r+1}^{s+t-2-p} (i_j-1)^{-E_{\sigma}(j)}~.
\]
This may be done with the help of the next two lemmas.

\begin{lemma}\label{lemma:smallexponents}
  Fix a vertex $v$, vertices
  $i_1<\cdots<i_r<v<i_{r+1}<\cdots<i_{s+t-p-2}$ and an ordering
  $\sigma$ of a double cycle on this set of vertices with $v$ as an
  anchor. Then, for every $k\in[r]$ we have
$$ k-1 \geq \sum_{i=1}^k E_{\sigma}(i)~. $$
\end{lemma}

\begin{proof}
  For $k\in [r]$, we define $G(k)$ as the subgraph of the double cycle
  in which we only keep the $k$ vertices of smallest index, so that
  $ \sum_{i=1}^k E_{\sigma}(i)$ is the number of edges in $G(k)$.

  Since $G(k)$ does not contain $v$, there are no cycles in $G(k)$,
  and therefore it is a forest. Since $|G(k)|=k$, it follows
  that $G(k)$ has at most $k-1$ edges.
\end{proof}

\begin{lemma}\label{lemma:bigexponents}
  Fix a vertex $v$, vertices
  $i_1<\cdots<i_r<v<i_{r+1}<\cdots<i_{s+t-p-2}$ and an ordering
  $\sigma$ of a double cycle on this set of vertices with $v$ as an
  anchor. Then, $\forall k \in [s+t-2-p-r]$ we have
$$k+1\leq \sum_{i=1}^k E_{\sigma}(s+t-1-p-i)~.$$
\end{lemma}

\begin{proof}
  For $k\in [s+t-2-p-r]$, we define $G'(k)$ as the subgraph of the
  double cycle in which we only keep the $k$ vertices of largest
  index. Vertex $i_{s+t-2-p-k}$ has at least two neighbors in the
  double cycle. From the definition of $E_{\sigma}(s+t-p-1-k)$,
  $E_{\sigma}(s+t-p-1-k)$ is then at least $2$ minus the number of
   neighbors of $i_{s+t-2-p-k}$ in the double cycle with larger index. The
  number of such neighbors of $i_{s+t-2-p-k}$ is exactly the
  number of edges in $G'(k)$ minus the number of edges in
  $G'(k-1)$. Denoting $G'(k)=\left(V'(k),E'(k)\right)$, it leads to
  \[
    E_{\sigma}(s+t-p-1-k)\geq 2-\left(\#E'(k)-\#E'(k-1)\right)~,
\]    
implying
$$ \sum_{i=1}^k E_{\sigma}(s+t-1-p-i)  \geq  2k-\#E'(k)~. $$ 
Since $G'(k)$ does not contain $v$, it is a forest. Moreover
$|G'(k)|=k$ so $G'(k)$ has at most $k-1$ edges, which concludes the proof.
\end{proof}
We may decompose $A$ as follows:
$$ A=\sum_{i_r<v}(i_r-1)^{-E_{\sigma}(r)}\cdots \sum_{i_1<i_2}(i_1-1)^{-E_{\sigma}(1)}~. $$
From Lemma \ref{lemma:smallexponents} with $k+1$, we know that $-E_{\sigma}(1)\geq 0$, leading to
$$ \sum_{i_1<i_2}(i_1-1)^{-E_{\sigma}(1)}\leq( i_2-1)^{1-E_{\sigma}(1)}~, $$
which in turn leads to
$$ A\leq\sum_{i_r<v}(i_r-1)^{-E_{\sigma}(r)}\cdots \sum_{i_2<i_3}(i_2-1)^{1-E_{\sigma}(1)-E_{\sigma}(2)}~. $$
Once again, by Lemma \ref{lemma:smallexponents} with $k=2$, we have $1-E_{\sigma}(1)-E_{\sigma}(2)\geq 0$, leading to

$$ \sum_{i_2<i_3}(i_2-1)^{1-E_{\sigma}(1)-E_{\sigma}(2)} \leq (i_3 -1)^{2-E_{\sigma}(1)-E_{\sigma}(2)}~. $$
Iterating this scheme $r$ times, using Lemma \ref{lemma:smallexponents} at each step leads to
\begin{equation}\label{Aupperbound}
A\leq (v-1)^{r-\sum_{i=1}^r E_{\sigma}(i)}~.
\end{equation}
Similarly, we decompose $B$ as 
$$ B= \sum_{i_{r+1}>v} (i_{r+1}-1)^{-E_{\sigma}(r+1)} \cdots \sum_{i_{s+t-p-2}>i_{s+t-p-3}}(i_{s+t-p-2}-1)^{-E_{\sigma}(s+t-p-2)}~. $$
It follows from Lemma \ref{lemma:bigexponents} that
$E_{\sigma}(s+t-p-2)\geq2$, and therefore
$$ \sum_{i_{s+t-p-2}>i_{s+t-p-3}}(i_{s+t-p-2}-1)^{-E_{\sigma}(s+t-p-2)}\leq (i_{s+t-p-3}-1)^{1-E_{\sigma}(s+t-p-2)}~. $$
Following an analogous reasoning to the upper bound of $A$, iterating
this scheme $s+t-2-p-r$ times, using Lemma \ref{lemma:bigexponents} at
each step leads to
\begin{equation}\label{Bupperbound}
B\leq (v-1)^{s+t-2-p-r-\sum_{j=r+1}^{s+t-2-p}E_{\sigma}(j)}~.
\end{equation}
Substituting \eqref{Aupperbound} and \eqref{Bupperbound} into
\eqref{expectationcalculus}, we obtain
$$ \EXP C_{s,t}(v)    \leq \sum_{p=1}^{\lfloor s/2\rfloor} \ \  \sum_{r=0}^{s+t-p-2}   \ \  \sum_{\sigma\in \Pi_{s+t-2-p}} \ell^{s+t-p}  (v-1)^{-E_{\sigma}(0)}\times (v-1)^{s+t-2-p-r-\sum_{j=r+1}^{s+t-2-p}E_{\sigma}(j)} \times (v-1)^{r-\sum_{i=1}^r E_{\sigma}(i)}~. $$
Since 
$$ \sum_{j=0}^{s+t-2-p}E_{\sigma}(j)=s+t-p~, $$
we have
$$ \EXP C_{s,t}(v)   \leq \frac{1}{v^2} \sum_{p=1}^{\lfloor s/2\rfloor} \ \  \sum_{r=0}^{s+t-p-2}   \ \  \sum_{\sigma\in \Pi_{s+t-2-p}} \ell^{s+t-p}  ~, $$
leading to
\begin{align*}
\E\left[ C_{s,t}(v) \right]   \ & \leq \sum_{p=1}^{\lfloor s/2\rfloor} \ell^{s+t-p} (s+t-p-2)!(s+t-p-2)\frac{1}{v^2}  \\
& \leq 2\ell^{s+t}\frac{(s+t)!}{v^2}~.
\end{align*}
Finally, we plug this bound in \eqref{unionboundSM}:
\begin{align}
\PROB \left( \exists v\geq K \ : \ v \in S_m \right) \ \leq & \ \sum_{v\geq K} \sum_{s,t \leq m_{\epsilon}} 2\ell^{s+t}\frac{(s+t)!}{v^2} \\
& \leq 4\ell^{2m_{\epsilon}}\left(2m_{\epsilon}\right)! \frac{1}{K}~.
\end{align}
 Choosing $K=8\frac{1}{\epsilon}\ell^{2m_{\epsilon}}(2m_{\epsilon})!$
 concludes the proof of \eqref{eq:setissmall} and therefore Theorem
 \ref{thm:main1} follows.

\section{Proof of Theorem \ref{thm:main2}}
\label{sec:proofofthm2}

The proof of Theorem \ref{thm:main2} is analogous to that of Theorem
\ref{thm:main1}. In order to avoid repeating essentially the same
argument, we only highlight the differences in the proofs.

It is enough to
prove that, choosing
$m \ = \ m_{\epsilon} \ = \left\lceil
  (9+12/c)\log(1/\epsilon)\right\rceil$ one
has
\[
  \PROB\left\{ 1\in S_m \ \text{and} \ |S_m|\leq \frac{4}{\epsilon}
    \left(c+1\right)^{2m}(2m)! \right\}\geq 1-\epsilon~.
\]  
This follows if we prove that 
\begin{equation}
\label{eq:rootinset1}
 \PROB\left\{ 1 \in S_m \right\} \ge 1-\frac{\epsilon}{2}
\end{equation}
and
\begin{equation}
\label{eq:setissmall1}
 \PROB\left\{ |S_m| \le \frac{4}{\epsilon} \left(c+1\right)^{2m}(2m)! \right\} \ge 1-\frac{\epsilon}{2}
\end{equation}
both hold.

Recall that the uniform Cooper-Frieze model is the union of a uniform
random recursive tree $G_1$ and an inhomogeneous Erd\H{o}s-R\'enyi random
graph $G_2$  (with edges probabilities $\min(c/\max(i,j)-1,1)$).

Proving \eqref{eq:rootinset1} and \eqref{eq:rootinset} shares the same
basic argument. In order to show that the root vertex is an anchor of
a double cycle of size $(s,t)$ for some $s,t\le m$, one may show that,
with the desired probability, there exist at least two vertices $i,j$
with sufficiently small index such that the edges $(1,i)$ and $(1,j)$
are not present in the uniform random recursive tree but they are
present in the inhomogeneous Erd\H{o}s-R\'enyi random
graph $G_2$. This follows by similar concentration arguments
(for sums of independent Bernoulli random variables  and for the
height of a uniform random recursive tree) as 
in the proof of Theorem \ref{thm:main1}.



The proof of \eqref{eq:setissmall1} is
once again analogous to the proof of
\eqref{eq:setissmall}. We remind the reader than the main step of the
proof of Theorem \ref{thm:main1} relies on the fact that a uniform
random $\ell$-dag is dominated by an inhomogeneous Erdős–Rényi
random graph with of edges probabilities
$\ell/(\max(i,j)-1)$, as shown in Proposition \ref{lem:negativecorrelation}.
Using a similar
reasoning as in Proposition \ref{lem:negativecorrelation}, 
one may prove that a uniform Cooper-Frieze random graph is dominated
by an an inhomogeneous Erdős–Rényi random graph with edge probabilities
$(c+1)/(\max(i,j)-1)$. The remainder of the proof is exactly the same
as that of the proof of \eqref{eq:setissmall} and concludes the proof of Theorem
\ref{thm:main2}.

\section{Concluding remarks}

In this paper we addressed the problem of finding the first vertex in
dynamically growing networks, based on observing a present-day
snapshot of the unlabeled network. This problem has mainly been
studied for trees and the main purpose of the paper is to study
root finding in more complex networks. 
The main results show that in certain natural models it is possible to
construct confidence sets for the root vertex whose size does not
depend on the observed network. These confidence sets contain the
root vertex with high probability, and their size only depends on the
required probability of error. We prove this property in two models
of random networks, namely uniform $\ell$-dags and a simplified
model inspired by a general random network model of Cooper and
Frieze. In both models, the constructed confidence set contains all
vertices that are anchors of certain small subgraphs that we call
``double cycles.''

The paper leaves a number of questions open.  We conjecture that the
upper bounds obtained for the size of the confidence set are
suboptimal (as a function of the probability of error $\epsilon$). To
substantially improve on these bounds one may need to consider
``global'' measures, reminiscent to the centrality measures employed
in the case of root finding in recursive trees, as opposed to the
``local'' method proposed here. However, their use and
analysis appears substantially more challenging.

Deriving lower bounds for the size of the confidence
set is another interesting open question.

Another path for further research is to extend the network models
beyond the uniform ones considered in this paper. The most
natural extensions are preferential attachment versions of the
models.

We end by noting that the methodology based on double cycles
also works in a variant of the uniform Cooper-Frieze model in which
the uniform random recursive tree is removed. More precisely, one may
consider an inhomogeneous Erd\H{os}-R\'enyi random graph on the vertex
set $[n]$ with edge probabilities
$\min\left(c/(\max(i,j)-1),1\right)$, where $c>1$ is a
constant. In this case one may prove the following. 

\begin{tcolorbox}
\begin{theorem}
\label{thm:main3}
Let $c>1$ and let $G=G^{(n)}$ be an inhomogeneous Erd\H{os}-R\'enyi
random graph on $n$
vertices, with edge probabilities
$p_{i,j}=\min\left(c/(\max(i,j)-1),1\right)$.
Root estimation is possible in $G$. In
particular, there exist constants $c_0,c_1>0$, depending on $c$ only,
such that
one may take
\[
  K(\epsilon) \le \left( \frac{c_0}{\epsilon^{c_1}}\right)^{\frac{c_0}{\epsilon^{c_1}}}~.
\]
\end{theorem}
\end{tcolorbox}

The outline of the proof is similar to that of Theorems
\ref{thm:main1} and \ref{thm:main2}. The only difference is in the
proof that the root vertex is an anchor of a sufficiently small double
cycle. To prove this, we may write $G$ as the union of two independent
inhomogeneous Erd\H{os}-R\'enyi random graphs as follows. Let $k$ be
a sufficiently large integer (only depending on $\epsilon$). Then we
may define $G_1=([n],E_1)$ and $G_2=(n,[E_2])$ as independent inhomogeneous
Erd\H{os}-R\'enyi random graphs such that for all $1\le i < j\le n$,
\[
  \PROB\left\{ (i,j) \in E_1 \right\} = \left\{
    \begin{array}{ll}
      \frac{c}{k} & \text{if} \ j \le k \\
      0 & \text{otherwise}
    \end{array} \right.
\]
and
\[
  \PROB\left\{ (i,j) \in E_2 \right\} = \left\{
    \begin{array}{ll}
      \frac{p_{i,j} - \frac{c}{k}}{1-\frac{c}{k}} & \text{if} \ j \le k \\
      p_{i,j} & \text{otherwise}
    \end{array} \right.
\]
Clearly, $G=([n],E_1\cup E_2)$. The subgraph of $G_1$ induced by the
vertex set $[k]$ is a supercritical Erd\H{os}-R\'enyi random graph and
therefore, with high probability, it has a connected ``giant'' component of size
that is linear in $k$. Then one may easily show that, with high probability,
there are three edges in $G_2$ of the form $(1,i)$, where $i$ belongs
to the giant component. This is enough for vertex $1$ to be an anchor of a double cycle.

The rest of the proof is identical to that of Theorem \ref{thm:main1}.

\bibliographystyle{plainnat}

\end{document}